\documentclass [11pt,twoside,a4paper]{article}
\usepackage{amsfonts}
\usepackage{amsthm}
\usepackage{amsmath}
\usepackage{amstext}
\usepackage{amssymb}
\usepackage{mathrsfs}
\usepackage{amscd}
\usepackage{xypic}
\usepackage{epsf} 
\usepackage{graphicx} 
\usepackage{fancybox} 
\usepackage{color} 
\usepackage{fancyhdr}
\usepackage[hang,footnotesize]{caption2} 

\usepackage[
  top=1.5cm,
  headheight=1cm,
  body={132mm,205mm},
  centering
 ]{geometry}
\usepackage{tikz}

\def\mathcal{\mathscr}
\newfont{\aaa}{cmb10 at 19pt}
\newfont{\bbb}{cmb10 at 11pt}
\newtheorem{lem}{Lemma}[section]
\newtheorem{thm}[lem]{Theorem}
\newtheorem{cor}[lem]{Corollary}
\newtheorem{rem}{Remark}[]
\newtheorem{pro}[lem]{Proposition}
\newtheorem{example}{Example}[]
\newtheorem{defin}[lem]{Definition}
\def\v1{\vspace{1mm}}

\def\le{\leqslant}
\def\leq{\leqslant}
\def\ge{\geqslant}
\def\geq{\geqslant}

\newcommand{\beq}{\begin{equation}}
\newcommand{\eeq}{\end{equation}}
\newcommand{\bey}{\begin{eqnarray}}
\newcommand{\eey}{\end{eqnarray}}
\newcommand{\beyy}{\begin{eqnarray*}}
\newcommand{\eeyy}{\end{eqnarray*}}

\makeatletter
\def\@evenhead{
\vbox{\hbox to \textwidth {}{\hspace{0mm}{\footnotesize
\thepage}}{\hspace{9cm} {\footnotesize {Jianmin CHEN, Jiayi CHEN}}}
\protect\vspace{1truemm}\relax \hrule depth0pt height0.15truemm
width\textwidth}}
\def\@evenfoot{}
\def\@oddhead{\vbox{\hbox to \textwidth
{{\hspace{0cm}{\footnotesize Frobenius-Perron theory of representation-directed algebras}\hfill{\footnotesize
\thepage}}\hspace{0mm}}{} \protect\vspace{1truemm}\relax\hrule
depth0pt height0.15truemm width\textwidth}}
\def\@oddfoot{}
\makeatother

\begin{document}

\thispagestyle{empty}

\fancypagestyle{firststyle}
{
\renewcommand{\topmargin}{-9mm}
\fancyhead[lO,RE]{\footnotesize Front. Math. China \\
https:/\!/doi.org\\[3mm]
}
\fancyhead[RO,LE]{\scriptsize \bf 
} \fancyfoot[CE,CO]{}}
\renewcommand{\headrulewidth}{0pt}


\setcounter{page}{1}
\qquad\\[5mm]

\thispagestyle{firststyle}

\noindent{\aaa{Frobenius-Perron theory \\[2mm]
of representation-directed algebras}}\\[1mm]

\noindent{\bbb Jianmin CHEN$^1,$\quad Jiayi CHEN$^1$}\\[-1mm]

\noindent\footnotesize{1\ \
School of Mathematical Sciences,
Xiamen University, Xiamen 361005, China}\\[6mm]

\vskip-2mm \noindent{\footnotesize$\copyright$\,Higher Education
Press 2021} \vskip 4mm

\normalsize\noindent{\bbb Abstract}\quad We study the Frobenius-Perron dimension of representation-directed algebras and quotient algebras of canonical algebras of type $ADE$, prove that the Frobenius-Perron dimension of a representation-directed algebra is always zero and the Frobenius-Perron dimension of a quotient algebra of canonical algebras of type $ADE$ is 0 or 1. Moreover, we give a sufficient and necessary condition for a quotient algebra of a canonical algebra of type $ADE$ under which its Frobenius-Perron dimension is 0.\vspace{0.3cm}

\footnotetext{Received May 25, 2021; accepted August 16,
2021\\
\hspace*{5.8mm}Corresponding author: Jiayi CHEN, E-mail:
chenjiayi@stu.xmu.edu.cn}

\noindent{\bbb Keywords}\quad
Frobenius-Perron dimension, endofunctor, representation-directed algebra, canonical algebra\\[1mm]
{\bbb MSC2020}\quad 16G60, 16G70, 
18G15
\\[0.4cm]

\section{Introduction}
\label{xxsec0}

\subsection{Backgrounds}
The spectral radius (also called the Frobenius-Perron dimension) of a
matrix is an elementary and extremely useful invariant in linear algebra,
combinatorics, topology, probability and statistics. For instance, one can classify all the finite graphs which are simple and connected by applying the spectral radius to adjacency matrix of them \cite{DG}.

The Frobenius-Perron dimension of an object in a semisimple finite tensor (or
fusion) category was introduced by Etingof-Nikshych-Ostrik in 2005 \cite{ENO} (also see \cite{EG, EGO, N}). Since then
it has become an extremely useful invariant in the study of fusion categories and
representations of semismiple (weak and/or quasi-)Hopf algebras.

Recently, the Frobenius-Perron dimension of an endofunctor of a category was
introduced by the authors in \cite{CG1}. It can be viewed as a
generalization of the Frobenius-Perron dimension of an object in
a fusion category introduced by Etingof-Nikshych-Ostrik \cite{ENO}. It was
shown in \cite{CG1, CG2, ZZ} that the Frobenius-Perron dimension has
strong connections with the representation type of a category.

In the present paper we pay our attention
to the Frobenius-Perron theory of representation-directed algebras and quotient algebras of canonical algebras of type $ADE$. As we know, representation-directed algebras and canonical algebras get widespread attention in representation theory, e.g., \cite{K, KK, V} for representation-directed algebras and \cite{JP, PS} for canonical algebras, the algebras of finite type are representation-directed and the algebras of tame type are derived equivalent to canonical algebras. Therefore, it is significant to study representation-directed algebras and canonical algebras.

We will give several correlative inequalities of Frobenius-Perron dimension of a category in this paper, prove the Frobenius-Perron dimension of a representation-directed algebra is always zero and find a way to calculate the Frobenius-Perron dimension of a family of quotient algebras of canonical algebras of type $ADE$.

Throughout, let $\Bbbk$ be an algebraically closed field and everything
be over $\Bbbk$.

\subsection{Conventions}
\label{xxsec0.8}
\begin{enumerate}
\item[(1)]
Usually $Q$ means a quiver.
\item[(2)]
If $A$ is an algebra over the base field $\Bbbk$, then ${\rm mod}A$
denote the category of finite dimensional left $A$-modules.
\end{enumerate}

\bigskip

\subsection{Organization}

The paper is organized as follows. We provide background materials
in Section 2. Several inequations of Frobenius-Perron dimension are introduced in
Section 3. Section 4 and Section 5 calculate the Frobenius-Perron dimension of representation-directed algebras and quotient algebras of canonical algebras, respectively.
The main results in Section 4 and Section 5 are as following.

\begin{thm}\label{xxthm0.1}
Let $A$ be a representation-directed algebra and $\tau$ be the Auslander-Reiten translation of ${\rm mod}A$. Then we have \[{\rm fpd}_{A}(\tau)=0.\]
Moreover, $\{{\rm fpd}_{A}^n (E^m)\}_{n\geq 1, m\geq 0}$ is a collection consisting of 1 (when $m=0$) and 0 (when $m>0$).
\end{thm}

\noindent Where ${\rm fpd}_{A} (\tau)$ is the Frobenius-Perron dimension of the endofunctor $\tau$ (see Definition \ref{xxdef2.3}) and $\{{\rm fpd}_{A}^n (E^m)\}_{n\geq 1, m\geq 0}$ is the Frobenius-Perron theory of ${\rm mod}A$ (see Definition \ref{xxdef2.7}).

\begin{thm}\label{xxthm0.2}
Let $A$ be a quotient algebra of a canonical algebra (see Definition \ref{xxdef1.6}).
 Then we have \[{\rm fpd}_{A}(E^1)=0\ or\ 1.\] Moreover, ${\rm fpd}_{A}(E^1)=0$ if and only if all the paths is commutative when they have the same source and target. Here, ${\rm fpd}_{A}(E^1)$ is the Frobenius-Perron dimension of the abelian category ${\rm mod}A$ (see Definition \ref{xxdef2.7}).
\end{thm}

\noindent  We present some examples in Section 6.

\section{Preliminaries}
\label{xxsec1}

\subsection{Frobenius-Perron dimension of a matrix}
\label{xxsec1.2}
Let $A$ be an $n\times n$-matrix over complex numbers ${\mathbb C}$.
The {\it spectral radius} of $A$ is defined to be
$$\rho(A):=\max\{ |r_1|, |r_2|, \cdots, |r_n|\}$$
where $\{r_1,r_2,\cdots, r_n\}$ is the complete multi-set of eigenvalues
of $A$. When each entry of $A$ is a positive real number, $\rho(A)$
is also called the {\it Perron root} or the {\it Perron-Frobenius
eigenvalue} of $A$.

Let $A=(a_{ij})_{n\times n}$ be an $n\times n$-matrix with
entries $a_{ij}$ in $\overline{\mathbb R}:=
{\mathbb R}\cup \{\pm \infty\}$.
Define $A'=(a'_{ij})_{n\times n}$ where
$$a'_{ij}=\begin{cases} a_{ij} & a_{ij}\neq \pm \infty,\\
x_{ij} & a_{ij}=\infty,\\
-x_{ij} & a_{ij}=-\infty.
\end{cases}
$$

\begin{defin}\cite[Definition 1.2]{CG1}
\label{xxdef1.2}
Let $A$ be an $n\times n$-matrix with entries in $\overline{\mathbb R}$.
The {\it spectral radius} of $A$ is defined to be
$$
\rho(A):=\liminf_{{\text{all}}\; x_{ij}\to \infty} \; \rho(A')
\quad \in \overline{\mathbb R}.
$$
\end{defin}

\subsection{Frobenius-Perron dimension of endofunctors}

If ${\mathcal C}$ is a $\Bbbk$-linear category, then
${\rm Hom}_{\mathcal C}(M,N)$ is a $\Bbbk$-module for all
objects $M,N$ in ${\mathcal C}$. If ${\mathcal C}$ is
also abelian, then ${\rm Ext}^i_{\mathcal C}(M,N)$ are
$\Bbbk$-modules for all $i\geq 0$. Let ${\rm dim}$ be the
$\Bbbk$-vector space dimension.

Throughout the rest of the paper, let ${\mathcal C}$ denote a
$\Bbbk$-linear category. A functor between two $\Bbbk$-linear
categories is assumed to preserve the $\Bbbk$-linear structure.
For simplicity,
${\rm dim}(A,B)$ stands for ${\rm dim} {\rm Hom}_{\mathcal C}(A,B)$ for any
objects $A$ and $B$ in ${\mathcal C}$.

The family of finite sets of nonzero objects in ${\mathcal C}$
is denoted by $\Phi$ and the family of sets of $n$ nonzero objects
in ${\mathcal C}$ is denoted by $\Phi_n$ for each $n\geq 1$.
It is clear that $\Phi=\bigcup_{n\geq 1} \Phi_n$. We do not
consider the empty set as an element of $\Phi$.

\begin{defin}\cite[Definition 2.1]{CG1}
\label{xxdef2.1}
Let $\phi=\{X_1, X_2, \cdots,X_n\}$ be a finite subset of nonzero
objects in ${\mathcal C}$, namely, $\phi\in \Phi_n$.
\begin{enumerate}
\item[(1)]
An object $M$ in ${\mathcal C}$ is called a {\it brick}
if
\begin{equation}
\notag
{\rm Hom}_{\mathcal C}(M,M)=\Bbbk.
\end{equation}

\item[(2)]
$\phi\in \Phi$ is called a {\it brick set}   if each $X_i$ is a brick
and
$$\dim(X_i, X_j)=\delta_{ij}$$
for all $1\leq i,j\leq n$. The set of brick
$n$-object subsets is denoted by $\Phi_{n,b}$ . We write $\Phi_{b}=\bigcup_{n\geq 1}
\Phi_{n,b}$ . Define the {\it b-height} of ${\mathcal C}$ to be
$$h_b({\mathcal  C})=\sup\{n\mid \Phi_{n,b} \;
{\rm{ is\; nonempty}}\}.$$
\end{enumerate}
\end{defin}

For $\mathcal{M}=(M_{ij})$ with each entry $M_{ij}\in \mathrm{Mod}-\Bbbk$, we denote a matrix $\mathrm{dim}\  \mathcal{M}=(\mathrm{dim}\ M_{ij})$. Then we have the following definitions.

\begin{defin}\cite[Definition 2.3]{CG1}
\label{xxdef2.3}
Retain the notation as in Definition \ref{xxdef2.1}, assume there is an assignment $\zeta: \Phi\to \bigcup_{n\ge1}M_{n\times n}({\rm Mod}-\Bbbk)$ and for each $n\geq 1$, the restriction of $\zeta$ on $\Phi_n$ is $\zeta_n: \Phi_{n}\to M_{n\times n}({\rm Mod}-\Bbbk)$.
$\zeta$ satisfies the property: if $\phi_1$ is a subset of $\phi_2$,
then $\zeta(\phi_1)$ is a principal submatrix of $\zeta(\phi_2)$.

$(1)$ Define the {\it adjacency matrix} of $\phi\in \Phi_n$ to be $$A(\phi,\zeta):={\rm dim} \; (\zeta(\phi)).$$

$(2)$ Define the {\it nth Frobenius-Perron dimension} of  $\zeta$ to be $${\rm fpd}^n_{\mathcal{C}} (\zeta):=\sup_{\phi\in \Phi_{n,b}}\{\rho(A(\phi,\zeta))\}.$$

$(3)$ Define the {\it Frobenius-Perron dimension} of $\zeta$ to be $${\rm fpd}_{\mathcal{C}} (\zeta):=\sup_n \{{\rm fpd}^n(\zeta)\}.$$

\end{defin}

\begin{rem}
Let $\sigma$ be an endofunctor of $\mathcal{C}$, then we can define an assignment
\[
\Phi\longrightarrow \bigcup_{n\ge1}M_{n\times n}({\rm Mod}-\Bbbk)
\]\[
\{X_1,X_2,\cdots,X_n\}\mapsto ({\rm Hom}_{\mathcal{C}}(X_i,\sigma(X_j)))_{i,j=1}^n.
\]Denote the assignment also by $\sigma$, then we can define the Frobenius-Perron dimension of $\sigma$ by
Definition \ref{xxdef2.3}.
\end{rem}

\begin{example}
\label{xxex2.6}

Let ${\mathfrak A}$ be a $\Bbbk$-linear abelian category.
For each $m\geq 1$ and
$\phi=\{X_1,\cdots,X_n\}$, define
$$E^{m}: \phi\longrightarrow \left({\rm Ext}^m_{\mathfrak A}(X_i,X_j)\right)_{n\times n}.$$
By convention, let ${\rm Ext}^0_{\mathfrak A}(X_i,X_j)$
denote ${\rm Hom}_{\mathfrak A}(X_i,X_j)$.
Then, for each $m\geq 0$, one can define the Frobenius-Perron
dimension of $E^{m}$ by Definition \ref{xxdef2.3}. If $\mathfrak{A}$ is a hereditary abelian category, let $D^b(\mathfrak{A})$ be the bounded derived category of $\mathfrak{A}$, $[1]$ be the shift functor, then we can know from \cite[Theorem 3.5]{CG1} that the Frobenius-Perron dimension of $E^1$ in $\mathfrak{A}$ is equal to the Frobenius-Perron dimension of $[1]$ in $D^b(\mathfrak{A})$.

\end{example}

\begin{defin}\cite[Definition 2.7]{CG1}
\label{xxdef2.7}
Let ${\mathfrak A}$ be an abelian category. The
{\it Frobenius-Perron dimension} of ${\mathfrak A}$
is defined to be
$${\rm fpd} {\mathfrak A}:={\rm fpd}_{\mathfrak A} (E^1)$$
where $E^1:={\rm Ext}^1_{\mathfrak A}(-,-)$
is defined as in Example \ref{xxex2.6}(1). The
{\it Frobenius-Perron theory} of ${\mathfrak A}$ is the collection
$$\{{\rm fpd}^n_{\mathfrak A} (E^m)\}_{n\geq 1, m\geq 0}$$
where $E^m:={\rm Ext}^m_{\mathcal A}(-,-)$ is defined as in Example
\ref{xxex2.6}(1).

\end{defin}
For simplicity, ``Frobenius-Perron'' is abbreviated to ``fp''; If $A$ is an algebra, ${\rm fpd}_{{\rm mod}A}$ is abbreviated as ${\rm fpd}_A$.

\section{Inequations of fp-dimension}

In this section, we will prove several inequations of fp-dimension. First we give the following lemma which is a well-known fact in linear algebras.

\begin{lem}
\label{xxlem1.7}
\begin{enumerate}
\item[(1)]
Let $B$ be a square matrix with nonnegative entries and let
$A$ be a principal minor of $B$. Then $\rho(A)\leq \rho(B)$.
\item[(2)]
Let $A=(a_{ij})_{n\times n}$ and $B=(b_{ij})_{n\times n}$
be two square matrices such that $0\leq a_{ij}\leq b_{ij}$
for all $i,j$. Then $\rho(A)\leq \rho(B)$.
\end{enumerate}
\end{lem}



\begin{pro}
\label{prop1}
Let $A$ be an algebra, $\mathcal{I}$ be a two-sided ideal of $A$, and $B=A/\mathcal{I}$. Then ${\rm mod}B$ is a full subcategory of ${\rm mod}A$, and we have\[
{\rm fpd}_B (E^1)\le {\rm fpd}_A(E^1).
\]
\end{pro}

\begin{proof}
A linear space $M$ becomes a $B$-module when there is an algebraic homomorphism $B\rightarrow {\rm End}_{\Bbbk}M$. Since we have a natural algebraic homomorphism $\pi:A\rightarrow B, a\mapsto a+\mathcal{I}$, a $B$-module can always become a $A$-module.

Let $M,N$ be $B$-modules, $\varphi,\psi$ are the corresponding algebraic homomorphisms. A linear mapping $f:M\rightarrow N$ becomes a $B$-module homomorphism if and only if $f\circ\varphi(b)=\psi(b)\circ f$ for any $b\in B$. We know $M,N$ are also $A$-modules, $\varphi'=\varphi\circ \pi,\psi'=\psi\circ \pi$ are the corresponding algebraic homomorphisms. It is clear that $f\circ\varphi'(a)=f\circ\varphi(\pi(a))=\psi(\pi(a))\circ f=\psi'(a)\circ f$ for any $a\in A$, so $f$ is also a $A$-module homomorphism. Conversely, if $f$ is a $A$-module homomorphism, we have $f\circ\varphi'(a)=\psi'(a)\circ f$ for any $a\in A$. Since $\pi$ is surjective, for any $b \in B$, we can find an $a\in A$ such that $f\circ\varphi(b)=f\circ\varphi(\pi(a))=\psi(\pi(a))\circ f=\psi(b)\circ f$. It follows that $f$ is also a $B$-module homomorphism. Therefore, ${\rm mod}B$ can be seen a full subcategory of ${\rm mod}A$.

Since for arbitrary $B$-modules $M,N$, ${\rm Ext}^1_B(M,N)$ can be seen as a subspace of ${\rm Ext}^1_A(M,N)$, we have ${\rm dimExt}^1_B(M,N)\le {\rm dimExt}^1_A(M,N)$. If $\phi$ is a brick set in ${\rm mod}B$, it is also a brick set in ${\rm mod}A$. Then by Lemma \ref{xxlem1.7}, we have $\rho(A(\phi,E_B^1))\le \rho(A(\phi,E_A^1))$. Thus we get \[
\sup_{\phi\in \Phi_{{\rm mod}B}}\{\rho(A(\phi,E_B^1))\}\le \sup_{\phi\in \Phi_{{\rm mod}B}}\{\rho(A(\phi,E_A^1))\le \sup_{\phi\in \Phi_{{\rm mod}A}}\{\rho(A(\phi,E_A^1))\}
\]where $\Phi_{{\rm mod}A}$ means the set of brick sets in ${\rm mod}A$, $\Phi_{{\rm mod}B}$ means the set of brick sets in ${\rm mod}B$. It follows that ${\rm fpd}_B (E^1)\le {\rm fpd}_A(E^1)$.
\end{proof}

\begin{lem}\cite[Theorem 2.13, Ch. IV]{ASS}
\label{xxthm2.3}
Let $A$ be an algebra, $\tau$ be the Auslander-Reiten translate of ${\rm mod}A$ and $M, N$ be two $A$-modules in ${\rm mod}A$. Then there exist isomorphisms\[
{\rm Ext}_A^1(M,N)\cong D\underline{{\rm Hom}}_A(\tau^{-1}N,M)\cong D\overline{{\rm Hom}}_A(N,\tau M)
\]that are functorial in both variables, where $\underline{{\rm Hom}}_A(M,N)={\rm Hom}_A(M,N)/\mathcal{P}(M,N),$ $\overline{{\rm Hom}}_A(M,N)={\rm Hom}_A(M,N)/\mathcal{I}(M,N)$, $\mathcal{P}(M,N)$ (respectively, $\mathcal{I}(M,N)$) is a subset of ${\rm Hom}_A(M,N)$ consisting of all homomorphisms that factor through a projective (respectively, injective) $A$-module.
\end{lem}

\begin{pro}
\label{xxcor2.4}
Let $A$ be an algebra, Then we have\[
{\rm fpd}_A (E^1)\le {\rm fpd}_A(\tau).
\]
\end{pro}

\begin{proof}
By Lemma \ref{xxthm2.3}, for any two $A$-modules $M,N$, we have \[{\rm dim}{\rm Ext}_A^1(M,N)={\rm dim} D\overline{{\rm Hom}}_A(N,\tau M)\le {\rm dimHom} {\rm Hom}_A(N,\tau M).\] It follows that\[
{\rm fpd}_A (E^1)=\sup_{\phi\in \Phi_b}\{\rho(A(\phi,E^1))\}\le \sup_{\phi\in \Phi_{b}}\{\rho(A(\phi,\tau))={\rm fpd}_A(\tau)
\]where $\Phi_{b}$ means the set of brick sets in ${\rm mod}A$.
\end{proof}

\begin{lem}\cite[Lemma 5.2, Ch. VIII]{ASS}
\label{xxlem2.5}
Let $A$ be an algebra, $\mathcal{I}$ be a two-sided ideal of A, and $B=A/\mathcal{I}$. If $M$ is a $B$-module, then the Auslander-Reiten translate $\tau_BM$ of $M$ in ${\rm mod}B$ is a submodule of the Auslander-Reiten translate $\tau_AM$ of $M$ in ${\rm mod}A$.
\end{lem}

\begin{pro}
Let $A$ be an algebra, $\mathcal{I}$ be a two-sided ideal of A, and $B=A/\mathcal{I}$.  Then we have\[
{\rm fpd}_B (\tau_B)\le {\rm fpd}_A(\tau_A).
\]
\end{pro}

\begin{proof}
For two $B$-modules $M,N$, by Lemma \ref{xxlem2.5}, there is a monomorphism $i: \tau_BN\rightarrow \tau_AN$. Since the functor ${\rm Hom}_A(M,-)$ is left exact, the induced homomorphism ${\rm Hom} (M,i): {\rm Hom}_A(M,\tau_BN)\rightarrow {\rm Hom}_A(M,\tau_AN)$ is also a monomorphism. So we have \[{\rm dimHom}_B(M,\tau_BN)={\rm dimHom}_A(M,\tau_BN)\le{\rm dimHom}_A(M,\tau_AN).\] It follows that\[
{\rm fpd}_B (\tau_B)=\sup_{\phi\in \Phi_{{\rm mod}B}}\{\rho(A(\phi,\tau_B))\}\le \sup_{\phi\in \Phi_{{\rm mod}A}}\{\rho(A(\phi,\tau_A))={\rm fpd}_{A}(\tau_{A})
\]where $\Phi_{{\rm mod}A}$ means the set of brick sets in ${\rm mod}A$, $\Phi_{{\rm mod}B}$ means the set of brick sets in ${\rm mod}B$.
\end{proof}

\section{Fp-dimension of representation-directed algebras}
Let $A$ be an algebra. Recall that a {\it path} in ${\rm mod}A$ is a sequence\[
M_0\xrightarrow{f_1}M_1\xrightarrow{f_2}M_2\rightarrow\cdots\rightarrow M_{t-1}\xrightarrow{f_t}M_t,
\] where $t\ge1$, $M_0,M_1,\cdots,M_t$ are indecomposable $A$-modules and $f_1,\cdots,f_t$ are non-zero non-isomorphisms homomorphisms. A path in ${\rm mod}A$ is called a {\it cycle} if its source module $M_0$ is isomorphic to its target module $M_t$. An indecomposable $A$-module that lies on no cycle in ${\rm mod}A$ is called a {\it directing module}. An algebra is called {\it representation-directed} if every indecomposable $A$-module is directing.

\begin{lem}
\label{lem3.1}
Let $A$ be a representation-directed algebra, $\{M_1,\cdots,M_n\}$ be a set of indecomposable $A$-modules where $n$ is a positive integer. Then there exists a permutation $\sigma$ of $\{1,\cdots,n\}$ such that there is no path from $M_{i}$ to $M_{j}$ for $\sigma(i)\le\sigma(j)$.
\end{lem}

\begin{proof}
Using induction on $n$. If $n=1$, there is only one element $M_1$. The conclusion is tenable.

Now assume the conclusion holds for $n=k-1$. When $n=k$, there exists $i_0$, such that there is no path from $M_{i_0}$ to the other modules. Otherwise, for each $i$, there exists $j$ and a path from $M_i$ to $M_j$, then we can get a longer path\[
M_{i_1}\rightarrow \cdots\rightarrow M_{i_2}\rightarrow\cdots\rightarrow M_{i_3}\rightarrow \cdots
\]Since the set is finite, we can find $i_s=i_t$. And then we get a cycle which contradicts the assumption $A$ is a representation-directed algebra. Define $\sigma_1$ permuting $i_0$ with 1. By induction, we have $\sigma_2$ such that there is no path from $M_{\sigma_1(i)}$ to $M_{\sigma_1(j)}$ for $1<\sigma_2\sigma_1(i)\le\sigma_2\sigma_1(j)$. Let $\sigma=\sigma_2 \sigma_1$, then we get what we need.
\end{proof}

\begin{thm}
\label{thm1}
Let $A$ be a representation-directed algebra. Then we have \[{\rm fpd}_{A}(\tau)=0.\]
\end{thm}

\begin{proof}
For any two bricks $M_1,M_2$, if ${\rm Hom}_A(M_1,\tau M_2)\ne 0$, then there exists a non-zero non-isomorphism homomorphism $f:M_1\rightarrow \tau M_2$. Note that Auslander-Reiten sequence $\tau M_2\rightarrow \oplus_{i=1}^k N_i\rightarrow M_2$ gives a path $\tau M_2\rightarrow N_1\rightarrow M_2$, so we get a path $M_1\rightarrow\tau M_2\rightarrow N_1\rightarrow M_2$.

Assume $\phi=\{M_1,\cdots,M_n\}$ is a brick set. By Lemma \ref{lem3.1}, we have a permutation $\sigma$ of $\{1,\cdots,n\}$ such that for $\sigma(i)\le\sigma(j)$ there is no path from $M_{i}$ to $M_{j}$, thus ${\rm Hom}_A(M_i,\tau M_j)=0$. Therefore, the adjacency matrix of $\phi$ is a strictly upper triangular matrix in some order, so the spectral radius of the matrix is zero. It follows that ${\rm fpd}_{A}(\tau)=0$.
\end{proof}

\begin{cor}
\label{cor1}
Let $A$ be a representation-directed algebra. Then we have \[{\rm fpd}_{A}(E^1)=0.\]
\end{cor}

\begin{proof}
By Proposition \ref{xxcor2.4} and Theorem \ref{thm1}, we have $0\le {\rm fpd}_{A} (E^1)\le {\rm fpd}_{A}(\tau)=0$. Thus ${\rm fpd}_{A}(E^1)=0.$
\end{proof}

Let $\mathfrak{A}$ be an abelian category with enough projective objects or enough injective objects. For two objects $M,N$ in $\mathfrak{A}$, recall from  \cite{O} that the zero element in $\mathrm{Ext}_{\mathfrak{A}}^m(N,M)$ can be represented as \[
[0]=\begin{cases}
\begin{tikzpicture}
\node (1) at (0,0) {$0$};
\node (2) at (1.5,0) {$M$};
\node (3) at (3,0){$M$};
\node (5) at (4.5,0){$N$};
\node (6) at (6,0) {$N$};
\node (7) at (7.5,0) {$0.$};
\draw[->] (1) --node[above ]{} (2);
\draw[->] (2) --node[above ]{Id} (3);
\draw[->] (3) --node[above ]{0} (5);
\draw[->] (5) --node[above ]{Id} (6);
\draw[->] (6) --node[above ]{} (7);
\node (1) at (9,0) {$(m=2)$};
\end{tikzpicture}\\
\begin{tikzpicture}
\node (1) at (0,0) {$0$};
\node (2) at (1.5,0) {$M$};
\node (3) at (3,0){$M$};
\node (4) at (4.5,0){$0$};
\node (5) at (6,0){$\cdots$};
\draw[->] (1) --node[above ]{} (2);
\draw[->] (2) --node[above ]{Id} (3);
\draw[->] (3) --node[above ]{} (4);
\draw[->] (4) --node[above ]{} (5);
\end{tikzpicture}\\
\begin{tikzpicture}
\node (1) at (-1,0) {};
\node (1) at (0,0) {};
\node (2) at (1.5,0) {$0$};
\node (3) at (3,0){$N$};
\node (4) at (4.5,0){$N$};
\node (5) at (6,0){$0.$};
\draw[->] (1) --node[above ]{} (2);
\draw[->] (2) --node[above ]{} (3);
\draw[->] (3) --node[above ]{Id} (4);
\draw[->] (4) --node[above ]{} (5);
\node (1) at (8,0) {$(m\ge3)$};
\end{tikzpicture}
\end{cases}
\]

We have the following property.

\begin{lem}
\label{lem3.4}
If there exists a positive integer $m\ge2$ such that ${\rm Ext}_\mathfrak{A}^m(N,M)\ne0$, then we can find a path from $M$ to $N$.
\end{lem}

\begin{proof}
We use induction on $m$. If $m=2$, ${\rm Ext}^2(N,M)\ne0$ means there exists a non-zero exact sequence
\begin{equation}
\label{E1.1.2}
0\rightarrow M\xrightarrow{f} L_1\xrightarrow{g} L_2\xrightarrow{h} N\rightarrow0.
\end{equation}
 Let $K=\bigoplus_{i=1}^k K_i$ be the kernel of $L_2\rightarrow N$, where each $K_i$ is indecomposable. We can conclude that there is a $K_i$ such that ${\rm Ext}^{1}(K_i,M)\ne0$ and ${\rm Ext}^{1}(N,K_i)\ne0$. Otherwise, ${\rm Ext}^{1}(K_j,M)=0$ or ${\rm Ext}^{1}(N,K_j)=0$ for each $j=1,2,\cdots, k$. Without loss of generality, we assume that ${\rm Ext}^{1}(\bigoplus_{i=1}^r K_i,M)=0$ and ${\rm Ext}^{1}(N,\bigoplus_{i=r+1}^k K_i)=0$. Denote $\bigoplus_{i=1}^r K_i$ by $K_{\underline{r}}$ and $\bigoplus_{i=r+1}^k K_i$ by $K_{\overline{r}}$. Then there is a direct summand $L'_1$ of $L_1$ and isomorphism $\varphi:L_1\rightarrow L'_1 \oplus K_{\overline{r}}$ satisfying the following commutative diagram\[
 \begin{tikzpicture}
\node (1) at (0,0) {$0$};
\node (2) at (2,0) {$M$};
\node (3) at (4,0){$L_1$};
\node (6) at (6,0) {$K_{\underline{r}}\oplus K_{\overline{r}}$};
\node (7) at (8,0) {$0$};
\draw[->] (1) --node[above ]{} (2);
\draw[->] (2) --node[above ]{} (3);
\draw[->] (3) --node[above ]{} (6);
\draw[->] (6) --node[above ]{} (7);
\node (11) at (0,-2) {$0$};
\node (12) at (2,-2) {$M$};
\node (13) at (4,-2){$L'_1 \oplus K_{\overline{r}}$};
\node (16) at (6,-2) {$K_{\underline{r}}\oplus K_{\overline{r}}$};
\node (17) at (8,-2) {$0$};
\draw[->] (11) --node[above ]{} (12);
\draw[->] (12) --node[above ]{$\begin{pmatrix}
   	 f' \\
     0
   	\end{pmatrix}$} (13);
\draw[->] (13) --node[above ]{$\begin{pmatrix}
   	 g_1& \\
     &1
   	\end{pmatrix}$} (16);
\draw[->] (16) --node[above ]{} (17);
\draw[->] (3) --node[left ]{$\varphi$} (13);
\draw[-] (2,-.3) --node[above ]{} (2,-1.7);
\draw[-] (2.1,-.3) --node[above ]{} (2.1,-1.7);
\draw[-] (6,-.3) --node[above ]{} (6,-1.7);
\draw[-] (6.1,-.3) --node[above ]{} (6.1,-1.7);
\end{tikzpicture}\]
 Similarly there is a direct summand $L'_2$ of $L_2$ and isomorphism $\psi:L_2\rightarrow  K_{\underline{r}}\oplus L'_2$ satisfying the following commutative diagram
 \[
 \begin{tikzpicture}
\node (1) at (0,0) {$0$};
\node (2) at (2,0) {$K_{\underline{r}}\oplus K_{\overline{r}}$};
\node (3) at (4,0){$L_2$};
\node (6) at (6,0) {$N$};
\node (7) at (8,0) {$0$};
\draw[->] (1) --node[above ]{} (2);
\draw[->] (2) --node[above ]{} (3);
\draw[->] (3) --node[above ]{} (6);
\draw[->] (6) --node[above ]{} (7);
\node (11) at (0,-2) {$0$};
\node (12) at (2,-2) {$K_{\underline{r}}\oplus K_{\overline{r}}$};
\node (13) at (4,-2){$K_{\underline{r}}\oplus K_{\overline{r}}$};
\node (16) at (6,-2) {$N$};
\node (17) at (8,-2) {$0$};
\draw[->] (11) --node[above ]{} (12);
\draw[->] (12) --node[above ]{$\begin{pmatrix}
   	 1 \\
     &g_2
   	\end{pmatrix}$} (13);
\draw[->] (13) --node[above ]{$\begin{pmatrix}
   	 0 &
     h'
   	\end{pmatrix}$} (16);
\draw[->] (16) --node[above ]{} (17);
\draw[->] (3) --node[right]{$\psi$} (13);
\draw[-] (2,-.3) --node[above ]{} (2,-1.7);
\draw[-] (2.1,-.3) --node[above ]{} (2.1,-1.7);
\draw[-] (6,-.3) --node[above ]{} (6,-1.7);
\draw[-] (6.1,-.3) --node[above ]{} (6.1,-1.7);
\end{tikzpicture}\]
Thus (\ref{E1.1.2}) is just the following exact sequence in the sense of isomorphism
\[
0\rightarrow M\xrightarrow{\begin{pmatrix}
   	 f' \\
     0
   	\end{pmatrix}} L'_1 \oplus K_{\overline{r}}\xrightarrow{\begin{pmatrix}
   	 g_1 \\
     &g_2
   	\end{pmatrix}} K_{\underline{r}}\oplus L'_2\xrightarrow{\begin{pmatrix}
   	 0 &
     h'
   	\end{pmatrix}} N\rightarrow0
\]Then we have commutative diagrams
\[
 \begin{tikzpicture}
\node (1) at (0,0) {$0$};
\node (2) at (2,0) {$M$};
\node (3) at (4,0){$M \oplus K_{\overline{r}}$};
\node (4) at (6,0){$K_{\underline{r}}\oplus L'_2$};
\node (6) at (8,0) {$N$};
\node (7) at (10,0) {$0$};
\draw[->] (1) --node[above ]{} (2);
\draw[->] (2) --node[above ]{$\begin{pmatrix}
   	 1 \\
     0
   	\end{pmatrix}$} (3);
\draw[->] (3) --node[above ]{$\begin{pmatrix}
   	 0 \\
     &g_2
   	\end{pmatrix}$} (4);
\draw[->] (4) --node[above ]{$\begin{pmatrix}
   	 0 &
     h'
   	\end{pmatrix}$} (6);
\draw[->] (6) --node[above ]{} (7);
\node (11) at (0,-2) {$0$};
\node (12) at (2,-2) {$M$};
\node (13) at (4,-2){$L'_1 \oplus K_{\overline{r}}$};
\node (14) at (6,-2){$K_{\underline{r}}\oplus L'_2$};
\node (16) at (8,-2) {$N$};
\node (17) at (10,-2) {$0$};
\draw[->] (11) --node[above ]{} (12);
\draw[->] (12) --node[above ]{} (13);
   \draw[->] (13) --node[above ]{} (14);	
\draw[->] (14) --node[above ]{} (16);
\draw[->] (16) --node[above ]{} (17);
\draw[->] (3) --node[right]{$\begin{pmatrix}
   	 f' \\
     &1
   	\end{pmatrix}$} (13);
\draw[-] (6,-.3) --node[above ]{} (6,-1.7);
\draw[-] (6.1,-.3) --node[above ]{} (6.1,-1.7);
\draw[-] (2,-.3) --node[above ]{} (2,-1.7);
\draw[-] (2.1,-.3) --node[above ]{} (2.1,-1.7);
\draw[-] (8,-.3) --node[above ]{} (8,-1.7);
\draw[-] (8.1,-.3) --node[above ]{} (8.1,-1.7);
\end{tikzpicture}\]
and
\[
 \begin{tikzpicture}
\node (1) at (0,0) {$0$};
\node (2) at (2,0) {$M$};
\node (3) at (4,0){$M \oplus K_{\overline{r}}$};
\node (4) at (6,0){$K_{\underline{r}}\oplus L'_2$};
\node (6) at (8,0) {$N$};
\node (7) at (10,0) {$0$};
\draw[->] (1) --node[above ]{} (2);
\draw[->] (2) --node[above ]{} (3);
\draw[->] (3) --node[above ]{} (4);
\draw[->] (4) --node[above ]{} (6);
\draw[->] (6) --node[above ]{} (7);
\node (11) at (0,-2) {$0$};
\node (12) at (2,-2) {$M$};
\node (13) at (4,-2){$M$};
\node (14) at (6,-2){$N$};
\node (16) at (8,-2) {$N$};
\node (17) at (10,-2) {$0$};
\draw[->] (11) --node[above ]{} (12);
\draw[->] (12) --node[above ]{1} (13);
   \draw[->] (13) --node[above ]{0} (14);	
\draw[->] (14) --node[above ]{1} (16);
\draw[->] (16) --node[above ]{} (17);
\draw[->] (3) --node[right]{$\begin{pmatrix}
   	1&0
   	\end{pmatrix}$} (13);
\draw[->] (4) --node[right ]{$\begin{pmatrix}
   	0&h
   	\end{pmatrix}$} (14);
\draw[-] (2,-.3) --node[above ]{} (2,-1.7);
\draw[-] (2.1,-.3) --node[above ]{} (2.1,-1.7);
\draw[-] (8,-.3) --node[above ]{} (8,-1.7);
\draw[-] (8.1,-.3) --node[above ]{} (8.1,-1.7);
\end{tikzpicture}\]
which implies that (\ref{E1.1.2}) is the zero element in ${\rm Ext}^2(N,M)$. So it contradicts to the assumption. Thus there exists a $K_i$ such that ${\rm Ext}^{1}(K_i,M)\ne0$ and ${\rm Ext}^{1}(N,K_i)\ne0$. Therefore, we get a path $M\rightarrow K_i\rightarrow N$.

Now assume that for $m=n-1$, the assertion holds. Consider $m=n$, then ${\rm Ext}^n(N,M)\ne0$ means there exists a nonzero exact sequence $0\rightarrow M\rightarrow L_1\rightarrow\cdots\rightarrow L_n\rightarrow N\rightarrow0$. Let $K=\bigoplus_{i=1}^k K_i$ be the kernel of $L_n\rightarrow N$, where each $K_i$ is indecomposable. Similarly, we can conclude that there exists a $K_i$ such that ${\rm Ext}^{n-1}(K_i,M)\ne0$ and ${\rm Ext}^1(N,K_i)\ne0$. By assumption, we get a path $M \rightarrow \cdots\rightarrow K_i$, then a path $M\rightarrow \cdots\rightarrow K_i\rightarrow N$.
\end{proof}

\begin{thm}
\label{thm2}
Let $A$ be a representation-directed algebra. Then we have \[{\rm fpd}_{A}(E^m)=0, m=2,3,\cdots.\]
\end{thm}

\begin{proof}
Let $\phi=\{M_1,\cdots,M_n\}$ be a brick set, $m\ge2$ be a positive integer. By Lemma \ref{lem3.1}, we have a permutation $\sigma$ of $\{1,\cdots,n\}$ such that for $\sigma(i)\le\sigma(j)$, there is no path from $M_{i}$ to $M_{j}$, thus ${\rm Ext}^m_A(M_j,M_i)=0$ by Lemma \ref{lem3.4}. Therefore, the adjacency matrix of $\phi$ is a strictly upper triangular matrix in some order, and $\rho(A(\phi,E^m))=0$. By the arbitrariness by $\phi$, we are done.

\end{proof}

\begin{cor}
Let $A$ be a representation-directed algebra. Then the Frobenius-Perron theory of ${\rm mod}A$ is a collection consisting of 1 and 0.
\end{cor}

\begin{proof}
For each brick set $\phi\in\Phi_{n,b}$, we have $A(\phi,E^{0})=\mathbb{E}_{n}$, where $\mathbb{E}_{n}$ is the identity matrix of order $n$. Thus ${\rm fpd}_{A}^n (E^0)=1$ for $n\geq 1$. Moreover, ${\rm fpd}_{A}^n (E^m)=0$ for $n\geq 1, m>0$ by Corollary \ref{cor1} and Theorem \ref{thm2}.
\end{proof}
\bigskip

\section{Fp-dimension of quotient algebras of a canonical algebra}

We consider a family of algebras closely related to canonical algebras.

\begin{defin}
\label{xxdef1.6}
Let $A_{\mathcal{I}}(n,m)=\Bbbk Q_A/\mathcal{I}$ for $n\ge1,m\ge0$, where $Q_A$ is the following quiver and $\mathcal{I}$ is an admissible ideal of $\Bbbk Q_A$.

\[\begin{tikzpicture}
\node (1) at (0,0) {1};
\node (2) at (1,1) {2};
\node (n) at (3,1) {n};
\node (n+1) at (1,-1) {n+1};
\node (n+m) at (3,-1) {n+m};
\node (n+m+1) at (4,0) {n+m+1};
\draw[->] (n+m+1) --node[right]{$\gamma_{m+1}$} (n+m);
\draw[->][dashed] (n+m) -- (n+1);
\draw[->] (n+1) --node[left]{$\gamma_1$} (1);
\draw[->] (n+m+1) --node[right]{$\alpha_{n}$} (n);
\draw[->] (2) --node[left]{$\alpha_1$} (1);
\draw[->][dashed] (n) -- (2);
\end{tikzpicture}\]

Let $D_{\mathcal{I}}(n)=\Bbbk Q_D/\mathcal{I}$ for $n\ge4$, where $Q_D$ is the following quiver and $\mathcal{I}$ is an admissible ideal of $\Bbbk Q_D$ satisfying $\alpha_1\cdots\alpha_{n-2}+\beta_1\beta_2+\gamma_1\gamma_2\in \mathcal{I}$.

\[
\begin{tikzpicture}
\node (1) at (0,0) {1};
\node (2) at (1,2) {2};
\node (n-2) at (3,2) {n-2};
\node (n-1) at (2,1) {n-1};
\node (n) at (2,-1) {n};
\node (n+1) at (4,0) {n+1};
\draw[->] (n+1) --node[above ]{$\beta_2$} (n-1);
\draw[->] (n+1) --node[below right]{$\gamma_2$} (n);
\draw[->] (n-1) --node[above ]{$\beta_1$} (1);
\draw[->] (n) --node[below left]{$\gamma_1$} (1);
\draw[->] (n+1) --node[above right]{$\alpha_{n-2}$} (n-2);
\draw[->] (2) --node[above left]{$\alpha_1$} (1);
\draw[->][dashed] (n-2) -- (2);
\end{tikzpicture}\]

Let $E_{\mathcal{I}}(n)=\Bbbk Q_E/\mathcal{I}$ for $n=6,7,8$, where $Q_E$ is the following quiver and $\mathcal{I}$ is an admissible ideal of $\Bbbk Q_E$ satisfying $\alpha_1\cdots\alpha_{n-3}+\beta_1\beta_2+\gamma_1\gamma_2\gamma_3\in \mathcal{I}$.

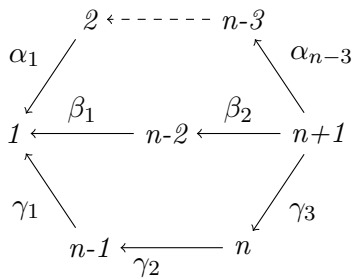
\[
\begin{tikzpicture}
\node (1) at (0,0) {1};
\node (2) at (1,1.5) {2};
\node (n-3) at (3,1.5) {n-3};
\node (n-2) at (2,0) {n-2};
\node (n-1) at (1,-1.5) {n-1};
\node (n) at (3,-1.5) {n};
\node (n+1) at (4,0) {n+1};
\draw[->] (n+1) --node[above]{$\beta_2$} (n-2);
\draw[->] (n+1) --node[below right]{$\gamma_3$} (n);
\draw[->] (n-2) --node[above ]{$\beta_1$} (1);
\draw[->] (n) --node[below left]{$\gamma_2$} (n-1);
\draw[->] (n-1) --node[below left]{$\gamma_1$} (1);
\draw[->] (n+1) --node[above right]{$\alpha_{n-3}$} (n-3);
\draw[->] (2) --node[above left]{$\alpha_1$} (1);
\draw[->][dashed] (n-3) -- (2);
\end{tikzpicture}
\]

Then $A_{\mathcal{I}}(n,m)$ for $n\ge1,m\ge0$, $D_{\mathcal{I}}(n)$ for $n\ge4$ and $E_{\mathcal{I}}(n)$ for $n=6,7,8$ are called {\it quotient algebras} of the canonical algebra of type $A, D$ and $E$, respectively.

\end{defin}

\begin{rem}
When $I$ is trivial, that is, $\mathcal{I}=\{0\}$ for $A_{\mathcal{I}}(n,m)$, $\mathcal{I}=\langle\alpha_1\cdots\alpha_{n-2}+\beta_1\beta_2+\gamma_1\gamma_2\rangle$ for $D_{\mathcal{I}}(n)$ or $\mathcal{I}=\langle\alpha_1\cdots\alpha_{n-3}+\beta_1\beta_2+\gamma_1\gamma_2\gamma_3\rangle$ for $E_{\mathcal{I}}(n)$, the corresponding quotient algebras $A_{\mathcal{I}}(n,m)$,$D_{\mathcal{I}}(n)$ and $E_{\mathcal{I}}(n)$ are just canonical algebras.
\end{rem}

The following lemmas are necessary.

\begin{lem}\cite[Theorem 0.3]{CG1}
\label{xxthm0.3}
Let $Q$ be a finite quiver and let ${\mathfrak A}$ be
the category of finite dimensional
left $\Bbbk Q$-modules.
\begin{enumerate}
\item[(1)]
$\Bbbk Q$ is of finite representation type
if and only if ${\rm fpd} {\mathfrak A}=0$.
\item[(2)]
$\Bbbk Q$ is of tame representation type
if and only if ${\rm fpd} {\mathfrak A}=1$.
\item[(3)]
$\Bbbk Q$ is of wild representation type
if and only if ${\rm fpd} {\mathfrak A}=\infty$.
\end{enumerate}
\end{lem}


\begin{lem}
\label{lem4.4}
Let $A=A_{\mathcal{I}}(n,m)$, if $\mathcal{I}=\langle\alpha_1\cdots\alpha_{n}+c\gamma_1\cdots\gamma_{m+1}\rangle$ for a nonzero constant $c$, then the Auslander-Reiten quiver of $A$ is connected by two Auslander-Reiten quiver with an indecomposable projective-injective module.
\end{lem}

\begin{proof}
Define $A_1=\Bbbk Q_{1}$ where $Q_1$ is a subquiver of $Q_A$ that the point $n+m+1$ is removed, $A_2=\Bbbk Q_{2}$ where $Q_2$ is a subquiver of $Q_A$ that the point $1$ is removed.

Notice that the left hand of the Auslander-Reiten quiver of ${\rm mod}A$ is equivalent to the left hand of ${\rm mod}A_{1}$ until the radical of $P(n+m+1)$ appears. So the left hand of the Auslander-Reiten quiver is \[\begin{tikzpicture}
\node (1) at (0,0) {$P(1)$};
\node (2) at (1,1) {$P(2)$};
\node (n) at (2,2) {$P(n)$};
\node (n+1) at (1,-1) {$P(n+1)$};
\node (n+m) at (2,-2) {$P(n+m)$};
\node (c)at(5,2){$S(n)$};
\draw[<-][dashed] (n+m) -- (n+1);
\draw[dashed] (c) -- (n);
\draw[<-] (n+1) --  (1);
\draw[<-] (2) --  (1);
\draw[<-][dashed] (n) -- (2);
\node (a) at (3,0) {$M(1)$};
\node (b) at (4,1) {$M(2)$};
\node (d) at (4,-1) {$M(3)$};
\node (e) at (5,-2) {$S(n+m)$};
\draw[<-][dashed] (e) -- (d);
\draw[<-] (d) --  (a);
\draw[<-] (b) --  (a);
\draw[<-][dashed] (c) -- (b);
\draw[dashed] (b) -- (2);
\draw[dashed] (a) -- (1);
\draw[dashed] (d) -- (n+1);
\draw[dashed] (e) -- (n+m);
\end{tikzpicture}\]
where $P(i)$ is the projective module at $i$, $S(j)$ is the simple module at $j$, $M(1),M(2),M(3)$ is respectively defined as follows\[\begin{tikzpicture}
\node (1) at (0,0) {$\Bbbk $};
\node (2) at (0.2,0.2) {$\Bbbk $};
\node (n) at (1,0.2) {$\Bbbk $};
\node (n+1) at (.2,-.2) {$\Bbbk $};
\node (n+m) at (1,-.2) {$\Bbbk $};
\node (n+m+1) at (1.2,0) {0};
\draw[dashed] (n+m) -- (n+1);
\draw[dashed] (n) -- (2);
\end{tikzpicture},\begin{tikzpicture}
\node (1) at (0,0) {$0$};
\node (2) at (0.2,0.2) {$\Bbbk $};
\node (n) at (1,0.2) {$\Bbbk $};
\node (n+1) at (.2,-.2) {$0$};
\node (n+m) at (1,-.2) {$0$};
\node (n+m+1) at (1.2,0) {0};
\draw[dashed] (n+m) -- (n+1);
\draw[dashed] (n) -- (2);
\end{tikzpicture},\begin{tikzpicture}
\node (1) at (0,0) {$0$};
\node (2) at (0.2,0.2) {$0$};
\node (n) at (1,0.2) {$0$};
\node (n+1) at (.2,-.2) {$\Bbbk $};
\node (n+m) at (1,-.2) {$\Bbbk $};
\node (n+m+1) at (1.2,0) {0.};
\draw[dashed] (n+m) -- (n+1);
\draw[dashed] (n) -- (2);
\end{tikzpicture}\] Similarly, the right hand of the Auslander-Reiten quiver of ${\rm mod}A$ is
\[\begin{tikzpicture}
\node (1) at (0,0) {$M(4)$};
\node (2) at (-1,1) {$M(2)$};
\node (n) at (-2,2) {$S(2)$};
\node (n+1) at (-1,-1) {$M(3)$};
\node (n+m) at (-2,-2) {$S(n+1)$};
\node (c)at(1,2){$I(n)$};
\draw[->][dashed] (n+m) -- (n+1);
\draw[dashed] (c) -- (n);
\draw[->] (n+1) --  (1);
\draw[->] (2) --  (1);
\draw[->][dashed] (n) -- (2);
\node (a) at (3,0) {$I(n+m+1)$};
\node (b) at (2,1) {$I(2)$};
\node (d) at (2,-1) {$I(n+m)$};
\node (e) at (1,-2) {$I(n+1)$};
\draw[->][dashed] (e) -- (d);
\draw[->] (d) --  (a);
\draw[->] (b) --  (a);
\draw[->][dashed] (c) -- (b);
\draw[dashed] (b) -- (2);
\draw[dashed] (a) -- (1);
\draw[dashed] (d) -- (n+1);
\draw[dashed] (e) -- (n+m);
\end{tikzpicture}\]
where $I(i)$ is the injective module at $i$, $S(j)$ is the simple module at $j$, $M(4)$ is defined as follow\[\begin{tikzpicture}
\node (1) at (0,0) {$0.$};
\node (2) at (0.2,0.2) {$\Bbbk $};
\node (n) at (1,0.2) {$\Bbbk $};
\node (n+1) at (.2,-.2) {$\Bbbk $};
\node (n+m) at (1,-.2) {$\Bbbk $};
\node (n+m+1) at (1.2,0) {$\Bbbk $};
\draw[dashed] (n+m) -- (n+1);
\draw[dashed] (n) -- (2);
\end{tikzpicture}\]
Note that $P(n+m+1)=I(1)=\begin{tikzpicture}
\node (1) at (0,0) {$\Bbbk $};
\node (2) at (0.2,0.2) {$\Bbbk $};
\node (n) at (1,0.2) {$\Bbbk $};
\node (n+1) at (.2,-.2) {$\Bbbk $};
\node (n+m) at (1,-.2) {$\Bbbk $};
\node (n+m+1) at (1.2,0) {$\Bbbk $};
\draw[dashed] (n+m) -- (n+1);
\draw[dashed] (n) -- (2);
\end{tikzpicture}$ is an indecomposable projective-injective module and the radical of $P(n+m+1)$ is $M_1$, so there exists a irreducible morphisma from $M(1)$ to $P(n+m+1)$. Similarly, there exists a irreducible morphisma from $I(1)$ to $M(4)$. Thus we get a Auslander-Reiten sequence.
\[\begin{tikzpicture}
\node (1) at (1.2,0) {$M(4)$};
\node (2) at (-1,1) {$M(2)$};
\node (n+1) at (-1,-1) {$M(3)$};
\node (b) at (-1,0) {$P(n+m+1)$};
\node (x) at (-3.2,0) {$M(1)$};
\draw[->] (n+1) --  (1);
\draw[->] (2) --  (1);
\draw[->] (x) --  (2);
\draw[->] (x) --  (n+1);
\draw[->] (x) --  (b);
\draw[->] (b) --  (1);
\end{tikzpicture}\]
Which implies that we can connect the left hand and the right hand of ${\rm mod}A$ by the indecomposable projective-injective module $P(n+m+1)$.
Thus every simple module is preprojective (also preinjective). By \cite[Proposition 4.7]{A}, we get the whole Auslander-Reiten quiver of ${\rm mod}A$.
\end{proof}

Now we calculate the Frobenius-Perron dimension of quotient algebras of the canonical algebra of type $A, D$ and $E$.

\begin{thm}
\label{xxthm4.3}
Let $A=A_{\mathcal{I}}(n,m)$.
 Then  \[{\rm fpd}_{A}(E^1)=0\ or\ 1.\] Moreover, ${\rm fpd}_{A}(E^1)=0$ if and only if there exists $0\ne c\in\Bbbk$ such that $\alpha_1\cdots\alpha_{n}+c\gamma_1\cdots\gamma_{m+1}\in\mathcal{I}$.
\end{thm}

\begin{proof}
(1) By Lemma \ref{lem4.4}, if $\mathcal{I}=\langle\alpha_1\cdots\alpha_{n}+c\gamma_1\cdots\gamma_{m+1}\rangle$ with a nonzero constant $c$, then $A$ is a representation-directed algebra. Therefore, by Corollary \ref{cor1}, ${\rm fpd}_{A}(E^1)=0$.

(2) If there exists a scalar $c\ne0$ such that $\alpha_1\cdots\alpha_{n}+c\gamma_1\cdots\gamma_{m+1}\in\mathcal{I}$, then $A$ is a quotient algebra of the algebra studied in (1). So by Proposition \ref{prop1}, we have ${\rm fpd}_{A}(E^1)=0$.

(3) If there is no scalar $c\ne0$ such that $\alpha_1\cdots\alpha_{n}+c\gamma_1\cdots\gamma_{m+1}\in\mathcal{I}$, we have $\alpha_1\cdots\alpha_{n}\ne0$ or $\gamma_1\cdots\gamma_{m+1}\ne0$ in $A$. Without loss of generality, assume $\gamma_1\cdots\gamma_{m+1}\ne0$. Consider the brick set $
\phi=\{S(2),S(3),\cdots,S(n),M(0)\},$ where $M(0)$ is defined as follow\[\begin{tikzpicture}
\node (1) at (0,0) {$\Bbbk$};
\node (2) at (0.2,0.2) {$0$};
\node (n) at (1,0.2) {$0$};
\node (n+1) at (.2,-.2) {$\Bbbk$};
\node (n+m) at (1,-.2) {$\Bbbk$};
\node (n+m+1) at (1.2,0) {$\Bbbk.$};
\draw[dashed] (n+m) -- (n+1);
\draw[dashed] (n) -- (2);
\end{tikzpicture}\]It is clear that ${\rm Ext}_{A}^{1}(S(i),S(i-1))\ne0$ for $3\le i\le n$, ${\rm Ext}_{A}^{1}(S(2),M(0))\ne0$ and ${\rm Ext}_{A}^{1}(M(0),S(n))\ne0$. Therefore we get\[
{\rm fpd}_{A}(E^1)\ge \rho\begin{pmatrix}
   	0 & 1&0&\cdots &0\\
   	0 & 0&1&\cdots&0\\
   \vdots&\vdots&\vdots&&\vdots\\
   0&0&0&\cdots&1\\
   1&0&0&\cdots&0
   	\end{pmatrix}=1.
\]Note that $A$ is a quotient algebra of a canonical algebra $\hat{A}$ which means that ${\rm fpd}_{A}(E^1)\le{\rm fpd}_{\hat{A}}(E^1)=1$ by Proposition \ref{prop1}. It follows that ${\rm fpd}_{A}(E^1)=1$.
\end{proof}

\begin{thm}
\label{xxthm4.4}
Let $D=D_{\mathcal{I}}(n)$.
 Then  \[{\rm fpd}_{D}(E^1)=0\ or\ 1.\] Moreover, ${\rm fpd}_{D}(E^1)=0$ if and only if there exists scalars $c_1,c_2\ne0$ such that $\alpha_1\cdots\alpha_{n-2}+c_1\beta_1\beta_2$ and $\alpha_1\cdots\alpha_{n-2}+c_2\gamma_1\gamma_{2}$ belong to $\mathcal{I}$.
\end{thm}

\begin{proof}
(1) If there exists scalars $c_1,c_2\ne0$ such that $\alpha_1\cdots\alpha_{n-2}+c_1\beta_1\beta_2$ and $\alpha_1\cdots\alpha_{n-2}+c_2\gamma_1\gamma_{2}$ belong to $\mathcal{I}$, then ${\rm fpd}_{D}(E^1)=0$. The proof is similar to Lemma \ref{lem4.4} and Theorem \ref{xxthm4.3}.

(2) If there are no scalars $c_1,c_2\ne0$ such that $\alpha_1\cdots\alpha_{n-2}+c_1\beta_1\beta_2,\alpha_1\cdots\alpha_{n-2}+c_2\gamma_1\gamma_{2}\in\mathcal{I}$. Let $\alpha=\alpha_1\cdots\alpha_{n-2}, \beta=\beta_1\beta_2$ and $\gamma=\gamma_1\gamma_{2}$. Then at least two of them do not belong to $\mathcal{I}$. Otherwise, there are two belonging to $\mathcal{I}$. Since $\alpha+\beta+\gamma$ belongs to $\mathcal{I}$, the other one also belongs to $\mathcal{I}$.  And then $\alpha+\beta,\beta+\gamma$ belong to $\mathcal{I}$ which contradicts the assumption.

Without loss of generality, assume $\beta,\gamma$ not belong to $\mathcal{I}$. In this case, neither of $\alpha+c\beta,\alpha+c\gamma$ belong to $\mathcal{I}$ for any scalar $c\ne0$. Consider the brick set $\phi=\{S(2),S(3),\cdots,S(n-3),M(0)\},$ where $M(0)$ is defined as follow\[\begin{tikzpicture}
\node (1) at (0,-.2) {$\Bbbk$};
\node (2) at (0.2,0.2) {$0$};
\node (n) at (1,0.2) {$0$};
\node (n+1) at (.6,0) {$\Bbbk$};
\node (n+m+1) at (1.2,-.2) {$\Bbbk$};
\node (a) at (.6,-.4) {$\Bbbk$};
\draw[dashed] (n) -- (2);
\end{tikzpicture}\]It is clear that ${\rm Ext}_{D}^{1}(S(i),S(i-1))\ne0,3\le i\le n-2$, ${\rm Ext}_{D}^{1}(S(2),M(0))\ne0$, ${\rm Ext}_{D}^{1}(M(0),S(n-2))\ne0$. Therefore we get\[
{\rm fpd}_{D}(E^1)\ge \rho\begin{pmatrix}
   	0 & 1&0&\cdots &0\\
   	0 & 0&1&\cdots&0\\
   \vdots&\vdots&\vdots&&\vdots\\
   0&0&0&\cdots&1\\
   1&0&0&\cdots&0
   	\end{pmatrix}=1.
\] It follows that ${\rm fpd}_{D}(E^1)=1$.

\end{proof}

\begin{thm}
Let $E=E_{\mathcal{I}}(n)$.
 Then we have \[{\rm fpd}_{E}(E^1)=0\ or\ 1.\] Moreover, ${\rm fpd}_{E}(E^1)=0$ if and only if there exists scalars $c_1,c_2\ne0$ such that $\alpha_1\cdots\alpha_{n-3}+c_1\beta_1\beta_2$ and $\alpha_1\cdots\alpha_{n-3}+c_2\gamma_1\gamma_{2}\gamma_3$ belong to $\mathcal{I}$.
\end{thm}

\begin{proof}
The proof is similar to Theorem \ref{xxthm4.4}.
\end{proof}

\section{Examples}
\label{xxsec5}
In this section we give some examples.

\begin{example}
    Let $A$ be a path algebra of a Dynkin quiver. As is known to us, $A$ is representation-directed, it follows that \[{\rm fpd}_{A}(E^{m})=0, m=1,2,3,\cdots..\]
\end{example}

\begin{example}
	Define quiver $Q$ as follow.
	
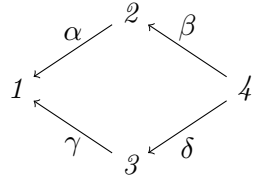
\[
	\begin{tikzpicture}
	\node (1) at (0,0) {1};
	\node (2) at (1.5,1) {2};
	\node (3) at (1.5,-1){3};
	\node (4) at (3,0){4};
	\draw[->] (2) --node[above ]{$\alpha$} (1);
	\draw[->] (3) --node[below ]{$\gamma$} (1);
	\draw[->] (4) --node[above]{$\beta$} (2);
	\draw[->] (4) --node[below ]{$\delta$} (3);
	\end{tikzpicture}
	\]
	
	(1) Let $A_1=kQ/(\alpha\beta-\gamma\delta)$, then the Auslander-Reiten quiver of ${\rm mod}A_1$ is as following.
		
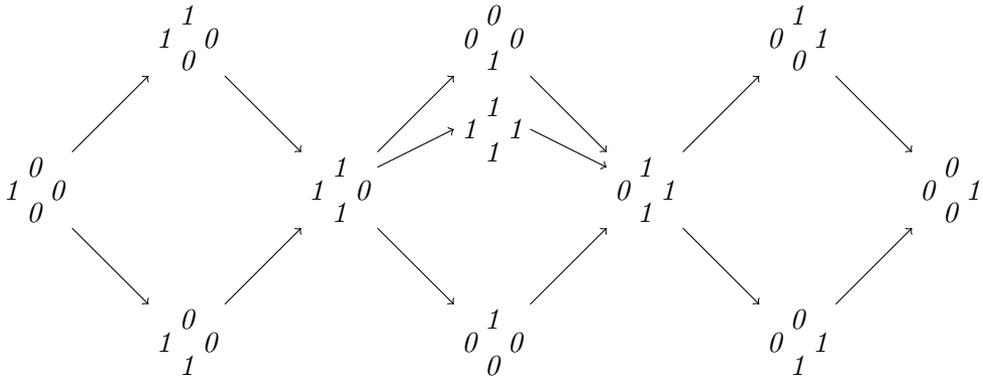
\[
	\begin{tikzpicture}
	\node (1) at (0+4,.3+4) {1};
	\node (2) at (0+4,-.3+4) {0};
	\node (3) at (.3+4,0+4){0};
	\node (4) at (-.3+4,0+4){0};
	\node (1) at (2+4,2.3+4) {1};
	\node (2) at (2+4,1.7+4) {1};
	\node (3) at (2.3+4,0+2+4){1};
	\node (4) at (1.7+4,0+2+4){0};
	\node (1) at (2+6,2.3+6) {1};
	\node (2) at (2+6,1.7+6) {0};
	\node (3) at (2.3+6,0+2+6){1};
	\node (4) at (1.7+6,0+2+6){0};
	
	\node (1) at (0,.3+4) {0};
	\node (2) at (0,-.3+4) {1};
	\node (3) at (.3,0+4){0};
	\node (4) at (-.3,0+4){1};
	\node (1) at (2,2.3+4) {1};
	\node (2) at (2,1.7+4) {1};
	\node (3) at (2.3,0+2+4){0};
	\node (4) at (1.7,0+2+4){1};
	\node (1) at (2+6-4,2.3+6) {0};
	\node (2) at (2+6-4,1.7+6) {1};
	\node (3) at (2.3+6-4,0+2+6){0};
	\node (4) at (1.7+6-4,0+2+6){0};
	
	\node (1) at (2-4,2.3+4) {0};
	\node (2) at (2-4,1.7+4) {0};
	\node (3) at (2.3-4,0+2+4){0};
	\node (4) at (1.7-4,0+2+4){1};
	\node (1) at (2+6-4-4,2.3+6) {1};
	\node (2) at (2+6-4-4,1.7+6) {0};
	\node (3) at (2.3+6-4-4,0+2+6){0};
	\node (4) at (1.7+6-4-4,0+2+6){1};

	\node (1) at (2+6-4,2.3+6-1.2) {1};
	\node (2) at (2+6-4,1.7+6-1.2) {1};
	\node (3) at (2.3+6-4,0+2+6-1.2){1};
	\node (4) at (1.7+6-4,0+2+6-1.2){1};
	
	\node (1) at (0+4+4,.3+4) {0};
	\node (2) at (0+4+4,-.3+4) {1};
	\node (3) at (.3+4+4,0+4){1};
	\node (4) at (-.3+4+4,0+4){0};
	\node (1) at (2+4+4,2.3+4) {0};
	\node (2) at (2+4+4,1.7+4) {0};
	\node (3) at (2.3+4+4,0+2+4){1};
	\node (4) at (1.7+4+4,0+2+4){0};
	
	\draw[->] (4.5,4.5) -- (5.5,5.5);
	\draw[->] (6.5,6.5) -- (7.5,7.5);
	\draw[->] (8.5,4.5) -- (9.5,5.5);
	\draw[->] (0.5,4.5) -- (1.5,5.5);
	\draw[->] (2.5,6.5) -- (3.5,7.5);
	\draw[->] (-1.5,6.5) -- (-0.5,7.5);

	\draw[->] (2.5,6.3) -- (3.5,6.8);
	\draw[->] (4.5,6.8) -- (5.5,6.3);
	
	\draw[->] (0.5,7.5) -- (1.5,6.5);
	\draw[->] (2.5,5.5) -- (3.5,4.5);
	\draw[->] (4.5,7.5) -- (5.5,6.5);
	\draw[->] (6.5,5.5) -- (7.5,4.5);
	\draw[->] (-1.5,5.5) -- (-.5,4.5);
	\draw[->] (8.5,7.5) -- (9.5,6.5);
	\end{tikzpicture}
	\]
	
Since $A_1$ is a representation-directed algebra, there is no cycle in ${\rm mod}A_1$. Thus we have \[{\rm fpd}({\rm mod}A_1)=0.\]
	
	(2) Let $A_2=kQ/(\alpha\beta)$, then the Auslander-Reiten quiver of ${\rm mod}A_2$ is as following.
	
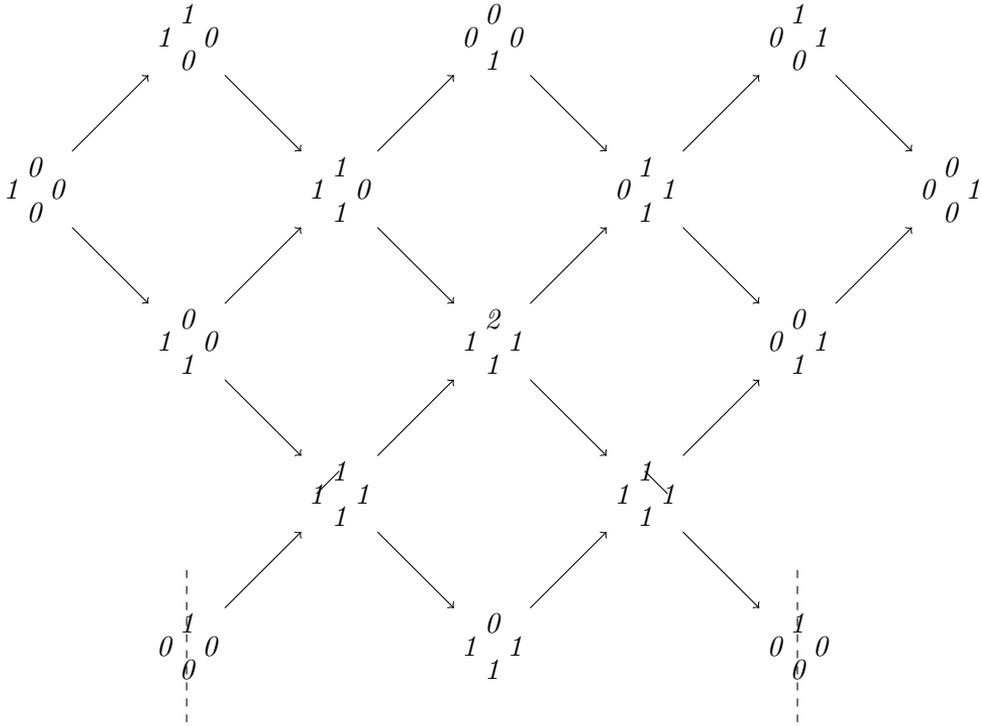
\[
	\begin{tikzpicture}
    \node (1) at (0+8,.3) {1};
    \node (2) at (0+8,-.3) {0};
    \node (3) at (.3+8,0){0};
    \node (4) at (-.3+8,0){0};
	\node (1) at (0,.3) {1};
	\node (2) at (0,-.3) {0};
	\node (3) at (.3,0){0};
	\node (4) at (-.3,0){0};
	\node (1) at (2,2.3) {1};
	\node (2) at (2,1.7) {1};
	\node (3) at (2.3,0+2){1};
	\node (4) at (1.7,0+2){1};
	\node (1) at (0+4,.3+4) {2};
	\node (2) at (0+4,-.3+4) {1};
	\node (3) at (.3+4,0+4){1};
	\node (4) at (-.3+4,0+4){1};
	\node (1) at (2+4,2.3+4) {1};
	\node (2) at (2+4,1.7+4) {1};
	\node (3) at (2.3+4,0+2+4){1};
	\node (4) at (1.7+4,0+2+4){0};
	\node (1) at (2+6,2.3+6) {1};
	\node (2) at (2+6,1.7+6) {0};
	\node (3) at (2.3+6,0+2+6){1};
	\node (4) at (1.7+6,0+2+6){0};
	
	\node (1) at (0,.3+4) {0};
	\node (2) at (0,-.3+4) {1};
	\node (3) at (.3,0+4){0};
	\node (4) at (-.3,0+4){1};
	\node (1) at (2,2.3+4) {1};
	\node (2) at (2,1.7+4) {1};
	\node (3) at (2.3,0+2+4){0};
	\node (4) at (1.7,0+2+4){1};
	\node (1) at (2+6-4,2.3+6) {0};
	\node (2) at (2+6-4,1.7+6) {1};
	\node (3) at (2.3+6-4,0+2+6){0};
	\node (4) at (1.7+6-4,0+2+6){0};
	
	\node (1) at (2-4,2.3+4) {0};
	\node (2) at (2-4,1.7+4) {0};
	\node (3) at (2.3-4,0+2+4){0};
	\node (4) at (1.7-4,0+2+4){1};
	\node (1) at (2+6-4-4,2.3+6) {1};
	\node (2) at (2+6-4-4,1.7+6) {0};
	\node (3) at (2.3+6-4-4,0+2+6){0};
	\node (4) at (1.7+6-4-4,0+2+6){1};
	
	\node (1) at (0+4,.3) {0};
	\node (2) at (0+4,-.3) {1};
	\node (3) at (.3+4,0){1};
	\node (4) at (-.3+4,0){1};
	\node (1) at (2+4,2.3) {1};
	\node (2) at (2+4,1.7) {1};
	\node (3) at (2.3+4,0+2){1};
	\node (4) at (1.7+4,0+2){1};
	\node (1) at (0+4+4,.3+4) {0};
	\node (2) at (0+4+4,-.3+4) {1};
	\node (3) at (.3+4+4,0+4){1};
	\node (4) at (-.3+4+4,0+4){0};
	\node (1) at (2+4+4,2.3+4) {0};
	\node (2) at (2+4+4,1.7+4) {0};
	\node (3) at (2.3+4+4,0+2+4){1};
	\node (4) at (1.7+4+4,0+2+4){0};
	
	\draw[->] (0.5,0.5) -- (1.5,1.5);
	\draw[->] (2.5,2.5) -- (3.5,3.5);
	\draw[->] (4.5,4.5) -- (5.5,5.5);
	\draw[->] (6.5,6.5) -- (7.5,7.5);
	\draw[->] (4.5,0.5) -- (5.5,1.5);
	\draw[->] (6.5,2.5) -- (7.5,3.5);
	\draw[->] (8.5,4.5) -- (9.5,5.5);
	\draw[->] (0.5,4.5) -- (1.5,5.5);
	\draw[->] (2.5,6.5) -- (3.5,7.5);
	\draw[->] (-1.5,6.5) -- (-0.5,7.5);
	
	\draw[->] (0.5,7.5) -- (1.5,6.5);
	\draw[->] (2.5,5.5) -- (3.5,4.5);
	\draw[->] (4.5,3.5) -- (5.5,2.5);
	\draw[->] (6.5,1.5) -- (7.5,0.5);
	\draw[->] (4.5,7.5) -- (5.5,6.5);
	\draw[->] (6.5,5.5) -- (7.5,4.5);
	\draw[->] (-1.5,5.5) -- (-.5,4.5);
	\draw[->] (0.5,3.5) -- (1.5,2.5);
	\draw[->] (2.5,1.5) -- (3.5,0.5);
	\draw[->] (8.5,7.5) -- (9.5,6.5);
	\draw[dashed](0,1) -- (0,-1);
	\draw[dashed](8,1) -- (8,-1);
	\draw(2,2.3) -- (1.7,2);
	\draw(6,2.3) -- (6.3,2);
	\end{tikzpicture}
	\]
	
There is a brick set consisting of$$
\begin{tikzpicture}
\node (0) at (-1,1.6+4) {,};
    \node (1) at (2-4,2.3+4) {1};
	\node (2) at (2-4,1.7+4) {0};
	\node (3) at (2.3-4,0+2+4){0};
	\node (4) at (1.7-4,0+2+4){0};
	\node (1) at (0,2.3+4) {0};
	\node (2) at (0,1.7+4) {1};
	\node (3) at (.3,0+2+4){1};
	\node (4) at (1.7-2,0+2+4){1};

\end{tikzpicture}
$$And the corresponding matrix is $\begin{pmatrix}
	0 & 1 \\
	1 & 0
	\end{pmatrix}
	$. Thus \[{\rm fpd}({\rm mod}A_2)=1.\]

\end{example}

\noindent\bf{\footnotesize Acknowledgements}\quad\rm{\footnotesize The authors would like to thank Yanan Lin and Shiquan Ruan for useful discussions, and thank the referees for their helpful comments. This work is supported by the National Natural Science Foundation of China (Grant Nos. 11971398 and 12131018).}\\[4mm]

\noindent{\bbb{References}}

\end{document}